\documentclass{article}
\usepackage{hyperref}
\usepackage[english]{babel}

\usepackage{graphicx}
\usepackage{amssymb}
\usepackage{amsthm, amsmath,amssymb}
\usepackage[all]{xy}
\usepackage{pstricks,pstricks-add,pst-node,pst-text,pst-3d}
\usepackage{graphicx}
\usepackage{psfrag}
\usepackage{multirow}

\usepackage[applemac]{inputenc}

\newcommand{\Rb}{\mathbb{R}}

\newcommand{\Zb}{\mathbb{Z}}
\newcommand{\Pb}{\mathbb{P}}
\newcommand{\mathsym}[1]{{}}
\newcommand{\unicode}[1]{{}}

\newtheorem{theorem}{Theorem}[section]

\newtheorem{lemma}{Lemma}
\newtheorem{proposition}{Proposition}

\theoremstyle{definition}
\newtheorem{definition}[theorem]{Definition}
\newtheorem{remark}{Remark}

\newtheorem{example}{Example}

\begin{document}

\title{Maps for global separation of roots}

\author{M{\' a}rio M. Gra{\c c}a\thanks{Departamento de Matem{\' a}tica,
Instituto Superior T{\' e}cnico, Universidade de Lisboa, Lisboa, Portugal. $mgraca@math.tecnico.ulisboa.pt$.   }
}
 
\maketitle

\begin{abstract}

\noindent
Two simple predicates are adopted and certain real-valued piecewise continuous functions are constructed from them.  This type of maps will be called quasi-step maps and aim to  separate the fixed points of an iteration map in an interval.  The main properties  of  these maps are studied. Several  worked examples are given where appropriate quasi-step maps for Newton and Halley iteration maps illustrate the main features of quasi-step maps as tools for global separation of roots.
\end{abstract}



\section{Introduction}\label{introd}

\noindent
Separation of real roots is a classical subject dating back  to the seminal work of Lagrange  \cite{La}  on polynomial equations. In this paper we aim to offer other computational  perspective to the global separation of roots of a general nonlinear equation by constructing certain iteration maps which will be called  {\em quasi-step} maps.

\noindent
We consider the  problem of finding the roots of a given real\--valued equation $f(x)=0$ on a closed interval $D=[a,b]$, which we write as  a  fixed point equation $x=g(x)$.  Let $ {\cal Z}=\{z_1,z_2,\ldots,z_n \}$ be  the non-empty  set of (distinct)  fixed points of the map $g$. If the set $ {\cal Z}$ was known, a good model of a map separating the fixed points of $g$ in the interval $D$ is
the  following  {\em step map} $\Psi$:
\begin{equation}\label{stepfA}
\Psi(x)=\sum_{i=1}^n z_i {\cal X}_{I_i}(x),\quad x\in D,
\end{equation}
 where   ${\cal X}_{I_i}$ is the \emph {characteristic function}   of the subinterval $I_i\subset D$ (i.e. ${\cal X}_{I_i}(x)=1$ for $x\in I_i$ and ${\cal X}_{I_i}(x)=0$ otherwise) and the intervals $I_i$ are  pairwise disjoint.

%
  \medskip
  \noindent
  In general, a step map of the form \eqref{stepfA} is not directly feasible  from $g$ since  the set of fixed points $ {\cal Z}$  is unknown. It is then natural to look for  a map $\widetilde \Psi$ of the  form
     \begin{equation}\label{stepfBB1}
\widetilde \Psi(x)=
\sum_{i=1}^s \widetilde g (x) {\cal X}_{J_i}(x),\qquad \text{$x\in D$}
\end{equation}
 where: 
 
 \noindent
 (a) The union of the intervals $J_i$  is contained in $D$ and each $J_i$ contains at least a fixed point of $g$.
 
 \noindent
  (b) The function  $ \widetilde g$ is  {\em continuous} on each subinterval $J_i$; 
  
  \noindent
  (c) The function   $ \widetilde g$   preserves the fixed points of   $g$ on each $J_i$.

 \medskip
 \noindent
  Like the map  \eqref{stepfA},  the map \eqref{stepfBB1} may be seen as a   tool for  separation of points in ${\cal Z}$.  A  map  $\widetilde \Psi$ of the type  \eqref{stepfBB1}  will be called a  {\em quasi-step map} (see Definition~\ref{edudef}).

\medskip
\noindent
The construction of a map $\widetilde g$ as  in  \eqref{stepfBB1} will be done by considering   one or more predicates which are based upon the map $g$ and the  domain $D$.  
The predicates to be used  hereafter are denoted by  $ \Pb_0$ and $ \Pb_1$. Given a constant $d>0$, these predicates are:
$$
\begin{array}{ll}
 \Pb_0:& x\in D,\, \, \text{ $y=g(x)\in D$  and $ |y-x|<d$},\\
 \Pb_1:&\text{$x\in D$,  $y=g(x)\in D$, $w=g(y)\in D$  and  $|w-y|\leq |y-x|$}.
 \end{array}
 $$
For any point $x\in D$ the predicate $ \Pb_0$ tests the image $y=g(x)$ while  $ \Pb_1$ tests two applications of  $g$.

\medskip
\noindent
A collection of subintervals $J_i\subset D$ is induced by a sort of {\em divide and conquer} effect from the action of one or both predicates.
Moreover when we use the predicate $\Pb_0$ (resp. $\Pb_1$) the value assigned to $\widetilde{g}(x)$ in \eqref{stepfBB1} will be $y=g(x)$ (resp. $w=g(g(x)))$ for all points $x\in D$ for which $\Pb_0$ (resp. $\Pb_1$)  is true. All the other points $x\in D$ for which the predicate under consideration is false,  the value zero will be  assigned to $\widetilde{g}$.  A  map $\widetilde \Psi$ as in \eqref{stepfBB1} constructed from one or more predicates will be called an \lq{educated}\rq \ map in the sense that the construction of this map is based on the action of the predicate(s).

\medskip
\noindent
 As explained in detail  in Section~\ref{secsep}, under mild assumptions on $g$,  the above predicates $ \Pb_0$ and $ \Pb_1$ lead to quasi-step maps of type \eqref{stepfBB1},  separating  fixed points of the initial map $ g$.  
As we will see, it is convenient to choose an initial map $g$ satisfying the property of attracting points in $D$ which are sufficiently close to the fixed points. Fortunately such a choice of maps $g$ does not presents any difficulty due to the plethora of iteration maps in the literature enjoying the referred attracting property.   Among them  we will consider  the ce\-le\-bra\-ted New\-ton\--Raphson and Halley maps  since as it is well known, under mild assumptions,  both maps  have at least linear local convergence  and so guaranteeing  the referred attracting property (\cite{traub}, \cite{deuflhard}, \cite{halley},\cite{ostrowski}, \cite{yamamoto}, \cite{verbeke}, and \cite{nerinckx}).  The proofs of how the predicates will lead to quasi-step maps separating fixed points of $g$, given in Section  \ref{secsep}, use mainly the fixed point theorem for closed and bounded real intervals  or the contraction Banach principle (see for instance \cite{ostrowski}, \cite{palais}, \cite{zeidler}). 
 
 \bigbreak
\noindent The  paper is divided in two parts. In the first part (Section \ref{secsep})  the main theoretical results are established and the second part deals with worked examples (Section \ref{worked}). In Section~\ref{secsep} we show under which  conditions on $g$, or on its fixed points, the predicates $\Pb_0$ and $\Pb_1$ will enable the construction of quasi-step maps providing a global separation of the fixed points of $g$ (propositions~\ref{propA} and \ref{cor2}).
 A brief reference on how the composition of quasi-step maps may be implemented to achieve  accurate approximations of  fixed points of $g$ is also made (see paragraph \ref{rfold}).
 
\medskip
\noindent
Section \ref{worked} is devoted to examples illustrating the separation of fixed points by constructing quasi-step maps from the predicates $ \Pb_0$ and $ \Pb_1$.
We  begin with a family of trigonometrical functions $f_k$ presented by Charles Pruitt  in \cite{mico}.  Due to the fact that Pruitt functions only admit integer zeros (cf. Proposition~\ref{proppruitt}) we are able to numerically construct (Example~\ref{exemplo1}) a step map which provides  not only  all the  zeros of $f_3$  but  it also enables to  distinguish composite numbers from prime numbers in the interval of definition of the map. 
Pruitt functions are also used  (Examples~\ref{exemplo2} and \ref{exemplo3}) to illustrate the main features of   several  quasi-step maps   derived from Newton and Halley maps educated by the predicates $\Pb_0$ and $\Pb_1$.   
\noindent
In Example~\ref{exemplo4}  a strongly oscillating transcendental function $f$ is considered.  Certain discretized  versions of  composed Halley educated maps  are applied in order to globally separate a great number of zeros of $f$ producing at the same time  accurate values for them.

 \section{ Separation of fixed points  }\label{secsep}
 
In this section we show how the predicates $\Pb_0$ and $\Pb_1$ will enable the construction of  quasi-step maps of type \eqref{stepfBB1} from a given function $g$.  Given an interval $D=[a,b]$ and a constant $d>0$,  
the predicates to be considered are the following:
\begin{align}\label{predp01}
 \Pb_0:&\;\; x\in D,\, \, \text{ $y=g(x)\in D$ and $ |y-x|<d$},\\ \label{predp11}
 \Pb_1:&\;\; \text{$x\in D$,  $y=g(x)\in D$, $w=g(y)\in D$  and  $|w-y|\leq |y-x|$}.
 \end{align}
 
 \noindent
Note that if $x$ is a fixed point of $g$, the predicates $\Pb_0$ and $\Pb_1$ hold true for $x$. Moreover,  assuming that $g$ is continuous near each fixed point, there are subintervals $J_i$ of $D$, containing fixed points of $g$, where the predicates hold true as well. We now  precise the notion of a quasi-step map.
  \begin{definition}\label{edudef}  
 Let $g:\Rb\rightarrow \Rb$ be a function and $D=[a,b]\subset \Rb$. A quasi-step map associate to $g$ is a function $\widetilde{\Psi}$ of the form
      \begin{equation}\label{stepfBB}
\widetilde \Psi(x)= \sum_{k=1}^s \widetilde g (x) \,{\cal X}_{J_k}(x),\quad \text{if $x\in D$}      
      \end{equation}
 where
 \begin{itemize}
 \item[(a)]  Any subinterval $J_k$ contains at least one fixed point of $g$ and  the union of these intervals  is contained in $D$ (i.e. $J=\cup_{k=1}^s J_k\subseteq D$); 
 \item[(b)]  The function  $ \widetilde g$ is  {\em continuous} on each subinterval $J_k$; 
 \item[(c)]    $ \widetilde g (z)= z$  for any fixed point $z$ of   $g$ belonging to $J_k$. 
 \end{itemize}
 \end{definition}
 
 \noindent
We note that if in the above definition all the subintervals $J_i$ are pairwise disjoint and the number $s$ coincides with the number $n$ of the fixed points of $g$, the quasi\--step map $\widetilde \Psi$ separates all the fixed points of $g$ in $D$.

\medskip
\noindent  For practical purposes we chose either one or both predicates $\Pb_0$, $\Pb_1$ to construct a quasi-step map  as in \eqref{stepfBB}. Such a  map will be called  \lq{educated}\rq\  by the predicate(s) in the sense that the map results from the action of the predicate(s) on the interval $D$.

 \begin{definition}\label{edudefA}  (Educated map)
 
\noindent
Let $g$ be a real\--valued map and consider the interval $D\subseteq \Rb$, ${\cal X}_{B}$ the characteristic function  of a set $B$, $d$ a positive constant, and  $\Pb_0$, $\Pb_1$  the predicates  in \eqref{predp01} and \eqref{predp11}, respectively. Let $\{L_k^i\}$ be the collection of subintervals of $D$ where the predicate $\Pb_i$ ($i=0,1$) holds true. 
 \begin{itemize}
\item[(a)] The  quasi-step map   
$$\widetilde \Psi_i(x)= \sum_{k} {\widetilde g}_i (x) \,{\cal X}_{L^i_k}(x), \quad x\in D$$
 is said to be educated by the predicate $\Pb_i$,  if ${\widetilde g}_i(x)=0$ for all $x\in D\setminus \cup_k L_k^i$,  and on each $L_k^i$:   ${\widetilde g}_i=g$ for $i=0$ and ${\widetilde g}_i=g\circ g$ for $i=1$.
\item[(b)] The quasi-step map $$\widetilde \Psi(x)= \sum_{k} {\widetilde g} (x) \,{\cal X}_{J_k}(x), \quad x\in D$$
 is said to be educated by both predicates $\Pb_0$ and $\Pb_1$, if ${\widetilde g}_i(x)=0$ for $x\in D\setminus \cup_k  (L^0_k\cap L^1_k)$ and 
 $ {\widetilde g}= g\circ g$ in each $J_k= L^0_k\cap L^1_k$.
\end{itemize}
\end{definition}

\noindent Note that a map  educated by the predicate $\Pb_0$ is a map which is necessarely zero at all the $x\in D$ such that $(x,g(x))$ does not lie in a  band of width  $2d$ centered at the line $y=x$. This is the reason  why we will call $d$  the {\em vertical displacement} parameter. On the other hand,  the predicate $\Pb_1$  tests    $y=g(x)\in D$ and $w=g(y)\in D$ satisfying $|w-y|\leq|y-x|$. Since,  for $y\neq w$  the quantity $(w-y)/(y-x)$ represents the slope of the line through the points $(x,y)$ and  $(y, w)$  we call  $\Pb_1$  the {\it slope predicate}.

\medskip
\noindent
We note that the \lq{education}\rq\  of a map is an {\it a priori} global technique (i.e. no initial guesses are required) which may be seen as a counterpart of the classical {\it a posteriori} stopping criteria used in root solvers algorithms for local search of roots (see for instance \cite{nikolajsen} and references therein). 

\noindent
Although in this work we only adopt the predicates $\Pb_0$ and $\Pb_1$, other predicates could be considered in order that the respective educated map will satisfy other criteria such as monotony or alternate local convergence. Also, an \lq{education}\rq\  of the map $f$ instead of the map  $g$ may be of interest, namely in the light of well known sufficient conditions for local monotone convergence of Newton and Halley methods (see for instance \cite{davies}, \cite{melman}).

\bigbreak\subsection{Quasi-step maps  from the predicates $\Pb_0$ and $\Pb_1$  }
 
 \noindent
 In this paragraph we address the question of knowing what kind of functions $g$ may be chosen in order to construct  educated maps from the predicates $\Pb_0$ and $\Pb_1$ leading to a global separation of fixed points of $g$. In particular,  with respect to the predicate $\Pb_0$ we will show that the contractivity of $g$ near each of its fixed points and  a conveniently chosen vertical displacement parameter $d$, provide  quasi-step maps isolating  fixed points of $g$ in $D$. 
In the case of the predicate $\Pb_1$, under mild hypotheses on $g$, this predicate implies the contractivity of $g$ near   fixed points and consequently enabling the construction of an educated map by  $\Pb_1$ isolating  fixed points of $g$ as well. 
 
 \medskip
 \noindent
 In what follows we assume that a real-valued  map $g$ is given, $D$ is the closed interval $D=[a,b]\subset \Rb$, and ${\cal X}_B$ denotes the characteristic function of the set $B$.   We denote by 
 ${\cal Z}\subset D$  the (non-empty) finite  set of the  fixed points of $g$, say $z_1,z_2,\ldots, z_n$.

\begin{lemma}\label{propA}
Let $d$ be a positive number and  $z$  a fixed point of  $g$ belonging to $D=[a,b]$. If $g$   is contractive  in the interval 
$$ I_z=[z-d,z+d]\subseteq D, $$
and the predicate $\Pb_0$ holds true for any point of $I_z$, then
 there exists a bounded closed interval
$$J_z=[z-\epsilon,z+\epsilon]\subseteq I_z,\quad \epsilon>0$$
such that the map 
$$\widetilde \Psi_z(x) = \widetilde{g}_z(x)\, {\cal X}_{J_z} (x)\;\;  x\in D,$$
with $\widetilde{g}_z(x)=0$ if $x \in D\setminus J_z$ and $ \widetilde{g}_z=g$ in $J_z$,
 isolates the (unique) fixed point $z$ of $g$  in $D$.
 \end{lemma}

\begin{proof}

\noindent First note that the hypothesis on the predicate $\Pb_0$  gives
\begin{equation}\label{sub2}
\begin{array}{l}
 |g(x)-x|< d,\quad x\in I_z.
\end{array}
\end{equation}

\noindent The points  $(x,g(x))$ of the plane satisfying the above  inequality  \eqref{sub2} belong to a closed planar region delimited by the parallel horizontal lines $y=z+d$, $y=z-d$ and the oblique parallel lines $y=x+d$, $y=x-d$. Denote by ${\cal P}$ the parallelogram bounding such  region and note that  the diagonals of ${\cal P}$  intersect  at the point $A=(z,z)$.

\noindent
Let  ${\cal D}$ be the closed disk of radius $r=d/\sqrt{2}$ centered at $A$. This disk  is inscribed in the region delimited by ${\cal P}$.  As by hypothesis  $g$ is contractive in $I_z$,  then $g$ is continuous in this interval. So, there exists $\delta>0$ such that
$$
x\in [z-\delta, z+\delta] \Longrightarrow |g(x)-z|<  r. 
$$
Also by the contractivity of  $g$,  there exists a number $K$, with  $0\leq  K<1$, ($K$ is a contractivity constant) such that  the graph of $g$ lies inside a cone section with vertex at  $A$ and  whose edges make an angle  $|\alpha|=\arctan(K)<\pi/4$  at the vertex $A$. Therefore, there exists a number  $0<\epsilon<r$ such that the square region ${\cal S}=[z-\epsilon]\times [z+\epsilon]\subset \Rb^2$ is contained in the disk ${\cal D}$.
 Moreover, for  the closed interval $J_z=[z-\epsilon,z+\epsilon]$ we have $g(J_z)\subseteq  J_z$. As $J_z$ is closed and  $g$ is contractive in  $J_z$, it follows from the Banach contraction principle that  there is a unique fixed point of $g$ in  $J_z$, which is obviously the point $z$. It is now immediate that the map $\widetilde \Psi_z$ isolates $z$ in $D$. 
\end{proof}

\noindent We now assume that the set ${\cal Z}$ of the fixed points of $g$ is ordered, and  define the {\em resolution} $\eta$ of  ${\cal Z}$ as the minimum of the distances between any pair of consecutive  points of ${\cal Z}$.

\begin{proposition}\label{cor1}
Let  ${\cal Z}=\{ z_1, \ldots, z_n    \}$ be the set of the fixed points of $g$ and  $\eta$ its resolution. Consider $g$ and $\Pb_0$ satisfying the conditions of Lemma~\ref{propA} on each interval $I_{z_i}=[z_i-d,z_i+d]$  with 
$$d<\frac{\eta}{2}.$$

\noindent
Then, there exists a collection of disjoint  subintervals $J_{z_i}\subset I_{z_i}$ with  $J=\cup_{i=1}^{n} J_{z_i}\subset D$, and maps $\widetilde{g}_{z_i}$ defined as    $\widetilde{g}_{z_i}(x)=0$, if $x\in D\setminus J$ and $\widetilde{g}_{z_i}=g$ in each interval $J_{z_i}$,
such that 
$$\widetilde \Psi (x)=
\sum_{i=1}^n \widetilde{g}_{z_i} (x)\, {\cal X}_{J_{z_i}} (x),\quad  \text{$x\in D$}
$$
is an educated map by $\Pb_0$,  separating all the fixed points of $g$ in $D$.

\end{proposition}

 \begin{proof}  The hypotheses on $g$ and $\Pb_0$ implies that Lemma~\ref{propA} is satisfied  for each fixed point $z_i$. That is:
  (i) there exists a closed subinterval $J_{z_i}=[z_i-\epsilon_{i},z_i+\epsilon_{i}]$ of $I_{z_i}= [z_i-d,z_i+d]\subset D$; 
  (ii) the map  $g$ is continuous in $J_{z_i}$; (iii) the fixed point $z_i$ is the unique fixed point of $g$ in  $J_{z_i}$.
  
  \noindent
  Hence,  as $d<\eta/2$ the subintervals $J_{z_i}$ are obviously disjoint and, by the definition of an educated map by $\Pb_0$ (cf. Definition \ref{edudef}), the map $\widetilde{\Psi}$ is an educated map separating the fixed points of $g$.
\end{proof}

\medskip
\noindent
We now explain the  role of the slope predicate $\Pb_1$  in the construction of a quasi-step map. 
Let us begin with a lemma  showing  that for a given fixed point $z$ verifying  $g'(z)\neq 1$ (i.e. a  {\em non-neutral} fixed point), if the predicate $\Pb_1$ holds in a certain interval containing $z$ it is possible to isolate this fixed point.

\begin{lemma}\label{propB}
Let $z\in D$ be a fixed point of a map $g$ which is of class $C^1$ in the interval  $X_z=[z-\epsilon, z+\epsilon]\subset D$ ($\epsilon>0$). Assume that  $g'(z)\neq 1$   and $\Pb_1$ holds true for all  non-fixed points of $g$ belonging to $ X_z$.
 Then, 
\begin{itemize}
\item[(i)] There is a closed bounded subinterval $\widetilde X_z=[z-\mu, z+\mu]$ of $X_z$ where the map $g$ is contractive.
\item[(ii)] The map
\begin{equation}\label{dec4}
{\widetilde\Psi}_z(x)=\widetilde{g}_{z}(x)\, {\cal X}_{\widetilde X_z} (x),\quad   \text{$x\in D$},
\end{equation}
with $\widetilde{g}_{z}(x)=0$ if $x\in D\setminus \widetilde X_z$ and $\widetilde{g}_{z}= g\circ g$ in $\widetilde X_z$, 
isolates the fixed point $z$. 
\end{itemize}
\end{lemma}
\begin{proof}

\noindent
{\it (i)} 
 Note that the inequality in $\Pb_1$   means that $| g(g(x)-g(x)  | \leq | g(x)-x  |   $, for all x in $I_z$ with $x\neq g(x)$. That is,
$$
\left | \displaystyle{ \frac{g(g(x))-g(x)}{g(x)-x}     }  \right|\leq 1 \, \Longleftrightarrow |q(x)|\leq 1,
$$
which implies:
\begin{equation}\label{dec10}
\lim_{x\rightarrow z} | q(x)| \leq 1.  
\end{equation}
Let $N(x)= g(g(x))-g(x)$ and $D(x)=g(x)-x$. Since $g\in C^1(X_z)$ both the functions $N$ and $D$ are of class $C^1$ in $X_z$.  As
\begin{equation}\label{dec8}
N'(x)= g'(x) (g'(x)-1)\quad\mbox{and}\quad   D'(x)=g'(x)-1,
\end{equation}
the  L'Hôpital's rule gives
$$
\lim_{x\rightarrow z} q(x)=  \lim_{x\rightarrow z}  \frac{N'(x)}{D'(x)}=  \lim_{x\rightarrow z}  \frac{g'(x)\, (g'(x)-1)}{g'(x)-1}=\lim_{x\rightarrow z}   g'(x)=g'(z)\neq 1,
$$
where the last two equalities follow from the continuity of $g'$ and the hypothesis $g'(z)\neq 1$.
 Hence, the inequality \eqref{dec10} holds strictly, which proves that $g$ is contractive in an interval centered at $z$, say $\widetilde X_z$.

\noindent
{\it (ii)}  
  Let $\widetilde X_z=[z-\mu,z+\mu]\subseteq X_z$ be  the closed interval $\widetilde X_z$ in item {\it (i)} where $g$ is contractive,  and $0\leq L<1$ a constant of contractivity. That is,
   \begin{equation}\label{dec3}
|g'(x)|\leq L<1,\quad \forall x\in \widetilde X_z.
\end{equation}
Let us prove that  $g(\widetilde X_z)\subseteq \widetilde X_z$. The mean value theorem for derivatives implies that  for all $x\in \widetilde X_z$, there exists $\theta\in (min\{x,z\},max\{x,z\})$ such that
\begin{equation}\label{dec5}
|g(x)-z|=|g'(\theta)|\, |x-z|\leq L\, \mu<\mu, \quad \forall x\in \widetilde X_z,
\end{equation}
where the inequalities  follows from \eqref{dec3}. The above inequality and the definition of   $\widetilde X_z$ imply  $g(\widetilde X_z)\subseteq \widetilde X_z$.
Moreover, as $g(\widetilde X_z)\subseteq \widetilde X_z$ and  $\widetilde X_z$ is a real closed bounded  interval, the  fixed point theorem for bounded closed  intervals guarantees that  $z$ is the unique fixed point of $g$ in $\widetilde X_z$. As $g(\widetilde X_z)\subseteq \widetilde X_z$ then $g(g(\widetilde X_z))\subseteq \widetilde X_z$,  and so the map ${\widetilde\Psi}_z$ in \eqref{dec4} isolates the fixed point $z$ in $D$.

\end{proof}

\noindent
We note that if  the hypotheses  of the above lemma are satisfied for all the fixed points of $g$ then we are able to separate all of them. The precise statement is as follows.
 
\begin{proposition}\label{cor2}
 Let ${\cal Z}=\{z_1,\ldots, z_n\}$ be the set of the fixed points of  a map $g$. Suppose that for each point $z_i$ all the assumptions of  Lemma~\ref{propB} are satisfied in the subintervals $X_{z_i}=[z_i-\epsilon_i, z_i+\epsilon_i]$.  Then, there exists a collection of disjoint subintervals  $K_i=[z_i-\delta, z_i+\delta]$ of $D$, and maps $\widetilde{g}_{z_i}$ with $ \widetilde{g}_{z_i} (x)=0$ if $x\in D\setminus \cup K_i$ and  $ \widetilde{g}_{z_i}= g\circ g  $ in $K_i$, such that the map
\begin{equation}\label{cor21}
\widetilde \Psi (x)=
 \sum_{i=1}^n \widetilde{g}_{z_i} (x)\, {\cal X}_{K_i} (x), \;\; \text{$x\in D$}
\end{equation}
 is an educated map by the predicate $\Pb_1$ separating all the fixed points of $g$.
\end{proposition}
\begin{proof}
By Lemma~\ref{propB} there are closed subintervals $\widetilde{X}_{z_i}=[z_i-\mu_i, z_i+\mu_i]$ each of them containing a unique  fixed point $z_i$ of $g$. Taking $\delta= \min\{\mu_1,\ldots,\mu_n\}$ and  $K_i=[z_i-\delta, z_i+\delta]\subset \widetilde{X}_{z_i}$,  the collection of the intervals $K_i$ is disjoint and the union of these intervals is contained in $D$. Also, by the proof of Lemma~\ref{propB} the function $g$ is continuous on each $K_i$ (due to the contractivity of $g$), and so $g\circ g$ is continuous on each $K_i$. Therefore the map $\widetilde \Psi$ in \eqref{cor21}  satisfies all the conditions of the definition of  an educated map by $\Pb_1$ and clearly separates all the fixed points $z_i$ of $g$.
\end{proof}
 
\subsubsection*{Composition of quasi\--step maps}\label{rfold}

 It is easy to see that the composition of a quasi-step map with itself is again a quasi-step map. Moreover, if the function $\widetilde{g}$ entering in the definition \eqref{stepfBB} of a quasi-step map $\widetilde \Psi$, is contractive in each interval $J_i$ it is natural that  compositions
of $\widetilde \Psi$ with itself will provide  better approximations of the fixed points of $g$ than  $\widetilde{\Psi}$.
 Since in Example~\ref{exemplo4} we use  compositions of Halley educated maps by $\Pb_1$, let us briefly explain some of their  features.

\noindent
Let  $g^r$ denote the $r$-fold composition of the map $g$ with itself, that is $g^r=(g\circ g\circ\cdots\circ g \circ g)$, where it is considered $r$ compositions of $g$ and  $g^1=g$.
 Consider as before ${\cal Z}=\{z_1, \ldots, z_n\}$ be the set of fixed points of a map $g$ and a quasi-step map 
     \begin{equation}\label{stepfBB1AA}
\widetilde \Psi(x)= \sum_{k=1}^n \widetilde g (x) \,{\cal X}_{J_k}(x),\quad \text{if $x\in D$}      
      \end{equation}
      with $\widetilde g$ a contractive map on each  interval $J_k=[z_k-\delta,z_k+\delta]\subset D$ and $\widetilde{g} (z_i)=z_i$, for all $i=1,\ldots, n$.
      
  \noindent Due to the contractivity of $\widetilde g$ on each $J_k$, there exists a constant $ 0\leq L<1$ such that   
  $$|\widetilde{g}^r(x)-z_k| \leq L^r |x-z_k|, \,\, x\in J_k.$$
This implies that the values of $\widetilde \Psi^r$ on each subinterval $J_k$ are closer to $z_k$ than the values of $\widetilde \Psi$ in the same interval.

 \noindent 
We remark that if $g$ is a map satisfying all the conditions of Proposition~\ref{cor2}, then the respective map  $\widetilde \Psi$ educated by the predicate $\Pb_1$, is of the form \eqref{stepfBB1} with $\widetilde{g}$ contractive on each interval $J_k$. Therefore, for $r>1$, the map $\widetilde \Psi^r$  should be considered a map,  in $D$,  closer to a step map  than the map $\widetilde{\Psi}$.
  
%
%
 
\noindent

  
\section{Examples}\label{worked}

\noindent
We  present several examples illustrating our procedures for global separation of roots. The first set of examples deal with a family of functions $f_k$ which we name  Pruitt  functions (see Definition \ref{pruittdef}), and the second set (Example \ref{exemplo4}) with  a strongly oscillating transcendental function.

\noindent
 In Example \ref{exemplo1} we obtain a step map for the Pruitt function $f_3$  and in the remaining  examples  appropriate quasi\--step maps are constructed based on the Newton and Halley maps  educated by  one or both the predicates $\Pb_0$ and $\Pb_1$. For convergence analysis and historical developments of the Newton and Halley methods we refer, for instance, to \cite{melman}, \cite{halley}, \cite{amat} and \cite{yamamoto}.

\medskip
\noindent
As before,  the equation $x=g(x)$ is a fixed point version of  $f(x)=0$. The set of fixed points of $g$, ${\cal Z}$, is generally unknown and we aim to find its elements in a given interval $D$. In all the examples no initial guesses of the fixed points are required. We will consider for $g$ either the Newton  ${\cal N}_f$ or the Halley ${\cal H}_f$ maps associated to $f$.  These maps are defined as follows.

\begin{definition}\label{def32} (Newton and Halley maps)

\noindent
Given a sufficiently differentiable map $f$, the Newton and Halley maps associate to $f$ are respectively:
\begin{equation}\label{eqhalley1}
\begin{array}{l}
{\cal N}_f(x)=x-f(x)/f'(x)\\
 {\cal H}_f(x)= x-\displaystyle{ \frac{ 2 f(x)f'(x)}{ 2\, f'(x)^2- f(x)\, f''(x)}} .
 \end{array}
 \end{equation} 
\end{definition}

\medskip

\noindent
It is well known (see  \cite{melman, ostrowski, proinov} and references therein) that if $f$ is sufficiently smooth on the interval  $D$, both ${\cal N}_f$  and ${\cal H}_f$ share the following properties: (a) for simple zeros of $f$, ${\cal N}_f$  and ${\cal H}_f$ have local order of convergence $p>1$; (b) for multiple zeros of $f$,   ${\cal N}_f$  and ${\cal H}_f$ have local linear convergence, (i.e. $p=1$). In particular, for simple zeros the order of convergence of ${\cal N}_f$ and of ${\cal H}_f$ is respectively $p\geq 2$ and $p\geq 3$.
 \noindent
In the light of our purposes,  both ${\cal N}_f$  and ${\cal H}_f$  satisfy a very convenient property concerning the zeros of $f$: (i) If $z$ is a simple zero of $f$, then $z$ is locally super attracting fixed point of  ${\cal N}_f$ and of ${\cal H}_f$. (ii) If $z$ is a zero of $f$ of higher multiplicity, then $z$ is  locally  an attracting fixed point of ${\cal N}_f$ and of ${\cal H}_f$. In both cases $g'(z)\neq 1$ and so one of the assumptions
 in Lemma~\ref{propB} is automatically satisfied.

\noindent
We note that besides ${\cal N}_f$  and ${\cal H}_f$ one could choose any other iteration map $g$ from the plethora of maps  in the literature, even with  greater order of convergence. For instance, the family of Halley maps presented in  \cite{wang} which has maximal order of convergence, or the recursive family of iteration maps in  \cite{MG} obtained from quadratures which has arbitrary order of convergence.
 
 \subsection{Step and quasi\--step maps for Pruitt's functions}\label{secpruitt}

\noindent
In this paragraph we consider a family of functions introduced by Pruitt in \cite{mico}  to illustrate our construction of step and quasi-step maps as tools for a global separation of roots. Due to the peculiar fact that these functions only admit integer roots,  we begin by showing that in this case  we are able  to present a step map which provides all the zeros of a Pruitt's function in a given interval.
  
\begin{definition}\label{pruittdef} (Pruitt function)

\noindent
  Let   $k$ be a positive integer. The $k$th Pruitt function $f_k$ is defined by
\begin{equation}\label{pruitt1}
f_k(x)= \prod_{i=1}^{k}  \sin\left(\displaystyle{\frac{x\, \pi}{p_i}}\right), \quad \mbox{for}\quad x\in I_k=[3/2,p_k^2+1/3],
\end{equation}
where $p_i$ denotes the $i$th prime number.
\end{definition}

%
%
 
 \subsubsection*{Step maps for Pruitt functions}

  \noindent
 In what follows we denote by $[x]$ the {\em closest integer} to $x\in\Rb$. The next proposition gives a step map of a Pruitt function.
  
\begin{proposition}\label{proppruitt}
Let $f_k$ be the $k$th Pruitt function defined on $I_k$ as in  \eqref{pruitt1},  $[\, \cdot\,]: I_k\rightarrow \Zb$ the closest integer function, and  $\Psi_k: I_k\rightarrow I_k$ the map defined by
\begin{equation}\label{proppruittA}
\Psi_k(x)=
\left\{
\begin{array}{ll}
[x],& \mbox{if}\quad f_k([x])=0\\
0,& \mbox{if}\quad f_k([x])\neq 0,
\end{array}
\right.
\end{equation}
 Then, 
\begin{itemize}
\item[(i)] The set ${\cal Z}$ of the fixed points  of  $\Psi_k$ coincides with  the set of  zeros of $f_k$, and it is given by   ${\cal Z}= {\cal P}_0\cup {\cal M}$, with 
$$
{\cal P}_0=\{ p: \,\, \text{ $p$ is prime and $2\leq p\leq p_k$} \},
$$
$$
{\cal M}=\{ s:\,\,  \text{$s$ is a multiple of an element of ${\cal P}_0$ and $ p_k< s \leq p_k^2$} \}.
$$
\item[(ii)] $\Psi_k$ is a step map in $I_k$ separating all  the zeros of $f_k$ in the interval $I_k$.
\item[(iii)] The  zeros of $\Psi_k$ which are integers are the primes $p$ such that $ p_k< p \leq p_k^2$.
\end{itemize}
\end{proposition}
\begin{proof}
\noindent
(i)-(ii)  Noting that the solutions of $\sin(\pi\,x/p)=0$ are $x=r\,p$ (with $r\in \Zb$), it follows from  the definition of $f_k$  that a zero of $f_k$  is an integer which belongs either to the set of primes ${\cal P}_0$ or to the set ${\cal M}$ of the multiples of the  elements in ${\cal P}_0$. As the fixed points of $\Psi_k$ are the integers of $I_k$ which are zeros of $f_k$, then the fixed point set of $\Psi_k$  is precisely ${\cal Z}={\cal P}_0 \cup{\cal M}$. 

\noindent
It is now obvious that \eqref{proppruittA}  can be written as the step map
$$\Psi_k(x)=\sum_{i=1}^{p_k^2-1} c_i\, {\cal X }_{I_i} (x), \quad x\in \Rb, $$
where $c_i$ is a non-negative integer and the intervals $I_i$ are defined as
$$
\begin{array}{ll}
I_i&=[i+1/2,i+3/2],\,\,\, \text{for $i=1,\ldots, p_k^2-2$}\\
I_{p_k^2-1}&= [p_k^2-1/2,p_k^2+1/3],
\end{array}
$$
with $I= \cup_{i=1}^{p_k^2-1} I_i$, which means that $\Psi_k$ separates the fixed points of $f_k$.

\noindent
(iii) From the proof of the previous items it is immediate that  the integers in $I_k\setminus {\cal Z}$ are the prime numbers $p$ belonging to $[p_k-1/2, p_k^2+1/3]$, that is $p_k\leq p<p_k^2$.
\end{proof}

\noindent
We remark that from the above proposition the zeros of $f_k$ in $I_k=[2,p_k^2]$ are  the first $p_k$ primes, $2, 3, \ldots, p_k$, and all their multiples which are less or equal to $p_k^2$. Furthermore,  any integer $j$  with   $p_k<j\leq p_k^2$, is a {\it composite} number if $f_k(j)=0$ and is a prime  number  if $f_k(j)\neq 0$.  
 
\medskip
\noindent
In the following example the step map given in  Proposition~\ref{proppruitt} is constructed for the 3th Pruitt function.
 \begin{example}\label{exemplo1}
\end{example}

 \begin{figure}[hbt] 
\begin{center}
 \includegraphics[scale=0.40]{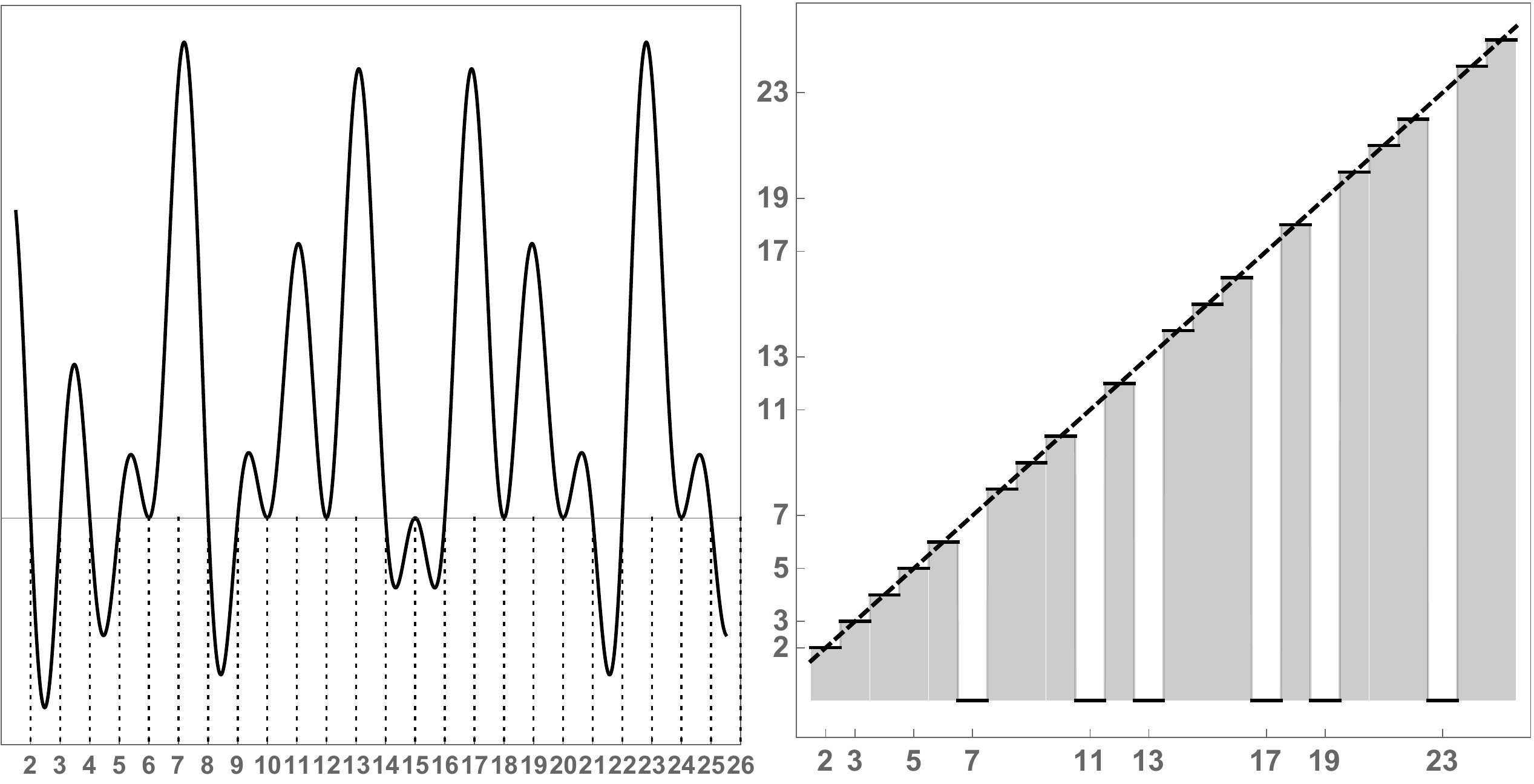}
 \caption{On the left the graphic of the Pruitt function $f_3$ on  $D=[1.5,25.5]$. On the right, the graphic of the step map $\Psi_3$ with the line $y=x$ (dashed). \label{fig1alfa}}
\end{center}
\end{figure}
\noindent
Let us consider $k=3$ and the  Pruitt function $f_3$:
$$
f_3(x)= \sin(x\, \pi/2)\, \sin(x\, \pi/3)\, \sin (x\, \pi/5),\qquad x\in I_3= [2-1/2,25+1/3].
$$
The system {\sl Mathematica}  \cite{wolfram1} has been used to compute the step map $\Psi_3$ defined    in \eqref{proppruittA}. For this prupose, the  built-in system function $Round[\, ]$ was used as the closest integer function.

\medskip
\noindent
The graphics of $f_3$ and $\Psi_3$ are displayed in Figure~\ref{fig1alfa}. Projecting the graph of $\Psi_3$ onto the $y$-axis we obtain the zeros of $f_3$. In fact, as predicted by Proposition~\ref{proppruitt}, this set is ${\cal Z}={\cal P}_0\cup {\cal M}$, with 
\begin{equation}\label{rootsf3}
{\cal P}_0= \{ 2,3,5 \},\quad {\cal M}= \{4,6,8,9,10,12,14,15,16,18,20, 21, 22,24,25\},
\end{equation}
where ${\cal P}_0$ is the set of the first $3$ primes and ${\cal M}$ the set of composite numbers which are multiples of the primes in ${\cal P}_0$.
Moreover,  by  Proposition~\ref{proppruitt}-{\it(iii)},   the step map $\Psi_3$ may also be used to detect the prime numbers greater than $5$. Its clear from the graphics of $\Psi_3$ in Figure~\ref{fig1alfa},  that these prime numbers are the integers (in the $x$-axis) which belong to those intervals where $\Psi_3$ is equal to zero. These primes are in the set ${\cal P}= \{7,11,13,17,19, 23\}$. 
 
 \begin{remark}\label{rem2} The sieve of Eratosthenes is an algorithm providing  the prime numbers less than a given integer $n\geq 2$. One of the versions of this algorithm computes precisely  the set of primes  ${\cal P}$ from the set ${\cal P}_0\cup {\cal M}$ in Proposition~\ref{proppruitt}. Thus,  the step function \eqref{proppruittA} may be seen as another computational version of the Eratosthenes algorithm. We refer to \cite{oneill} for a discussion of efficient practical versions  of the sieve of Eratosthenes using appropriate data structures.
 \end{remark}

 \subsubsection*{Halley and Newton quasi\--step maps for Pruitt functions}\label{subhalley}

In the examples  below we illustrate the construction of Halley and Newton quasi\--step maps, educated by the predicates  $\Pb_0$ and $\Pb_1$,
 for the Pruitt family. Although such construction does not depend on an {\em a priori} knowledge of the multiplicity of the zeros of the  function, we  note that for a Pruitt function $f_k$ we are dealing with both simple and multiple zeros in the interval $I_k$. In fact, the multiplicity of the zeros of $f_k$ satisfy an interesting property: on $I_k$ there are zeros of multiplicity $m$, where $m$ takes all the integer values $m=1,2,\ldots, k-1$. This property follows from the particular form of $f_k$, and although it can be proved in full generality we briefly explain it through a particular example, namely  with the  Pruitt function $f_4$.

\medskip
\noindent
The set of the first four primes is ${\cal P}_4=\{ 2,3,5,7  \}$.  Let ${\cal M}_1$ be the set of all the multiples of the elements of ${\cal P}_4$ which belong to the interval $I_4=[4-1/2, 49+1/3]$. That is,
$$
\begin{array}{ll}
{\cal M}_1&=\{2,3,4,5,6,7,8,9,10,12,14,15,16, 18,20,21,22,24,\\
&\hspace{5mm} 25,26,27,28,30,32,33,34,35,36,38,39,40,42,44,45,46,48,49\}.
\end{array}
$$
The elements of ${\cal M}_1$ are zeros of $f_4$ having multiplicity at least one. Consider now the $6=\binom{4}{2}$ numbers which are products of two elements of ${\cal P}_4$. These products are
$$
s_1=2\times 3, \, s_2=2\times 5, \, s_3=2\times 7, \, s_4=3 \times 5, \, s_5=3\times 7, \, s_6=5\times 7.
$$
Let ${\cal M}_2$  be the set of the multiples of $s_1$, $s_2, \ldots, s_6$ which belong to $I_4$:
$$
\begin{array}{ll}
{\cal M}_2&=\{6,10,12,14,15,18,20,21,24, 28,30,35,36,40,42,45,48\}.
\end{array}
$$
The elements of ${\cal M}_2$  are zeros of $f_4$ with multiplicity at least two. Thus, the simple zeros of $f_4$ belong to the set ${\cal Z}_1={\cal M}_1\verb+\+{\cal M}_2$, namely to
$${\cal Z}_1= \{ 2,3,4,5,7, 8,9,16,22,25,26,27,32,33,34,38,39,44,46,49  \}. $$
Taking the products of three elements of ${\cal P}_4 $ we obtain
$$
s_1=2\times 3\times 5, \, s_2=2\times 3\times 7, \, s_3=2\times 5\times 7, \, s_4=3 \times 5\times 7.
$$
The set ${\cal M}_3$  of all the multiples of $s_1$ to $s_4$ belonging to $I_4$ is
$$
\begin{array}{l}
{\cal M}_3=\{s_1,s_2\}=\{30,42\}.
\end{array}
$$
As all the elements of ${\cal M}_3$  have multiplicity greater or equal to three, then  ${\cal Z}_2={\cal M}_2\setminus {\cal M}_3$ is the set of double zeros of $f_4$:
$${\cal Z}_2= \{ 6,10,12,14, 15, 18, 20, 21, 24, 28, 35, 36, 40, 45, 48 \} \; \text{(double zeros)} $$
Finally, as the product of all the elements of  ${\cal P}_4$ is $2\times3\times5\times 7>49$, there are not zeros of multiplicity greater or equal to  four in $I_4$. This means that
$${\cal Z}_3={\cal M}_3=\{30,42\} \,\, \mbox{(triple zeros)}, $$
is the set of the  zeros of $f_4$ with multiplicity $m=3$ in the interval $I_4$.

\medskip
\noindent
Note that for the function $f_3$ treated in Example \ref{exemplo1}, either using the previous reasoning or just by the analysis of  the graphic of $f_3$ in Figure~\ref{fig1alfa}, we draw the following conclusions:
\begin{itemize}
\item The set of simple zeros of $f_3$ is:
\begin{equation}\label{es1}
\{  2,3,4,5,8,9,14,16,21,22,25 \}.
\end{equation}
\item The set of double zeros of $f_3$ is:
\begin{equation}\label{es2}
\{  6,10,12,15,18,20,24 \}.
\end{equation}
\end{itemize}

\medskip
\noindent The following two  examples illustrate several features of the   Halley and Newton quasi\--step maps, associated to Pruitt functions, educated by the predicates $\Pb_0$ and $\Pb_1$ in  \eqref{predp01} and \eqref{predp11} respectively. A quasi-step map educated by the predicate $\Pb_0$ will be denoted by $\widetilde{\Psi}_0$ and by $\widetilde{\Psi}_1$ in the case of $\Pb_1$.

 \begin{example}\label{exemplo2} (Halley quasi\--step maps for $f_3$)
 \end{example}
 
 \noindent
 Let $f_3$ be the 3th Pruitt function restricted to the interval $D=[1.5, 8.5]$ and $g={\cal H}_{f_3}$ the respective Halley map.  Let  $\widetilde{\Psi}_0$ be a Halley map  educated by the predicate $\Pb_0$.

   \begin{figure}[hbt] 
\begin{center}
  \includegraphics[scale=0.30]{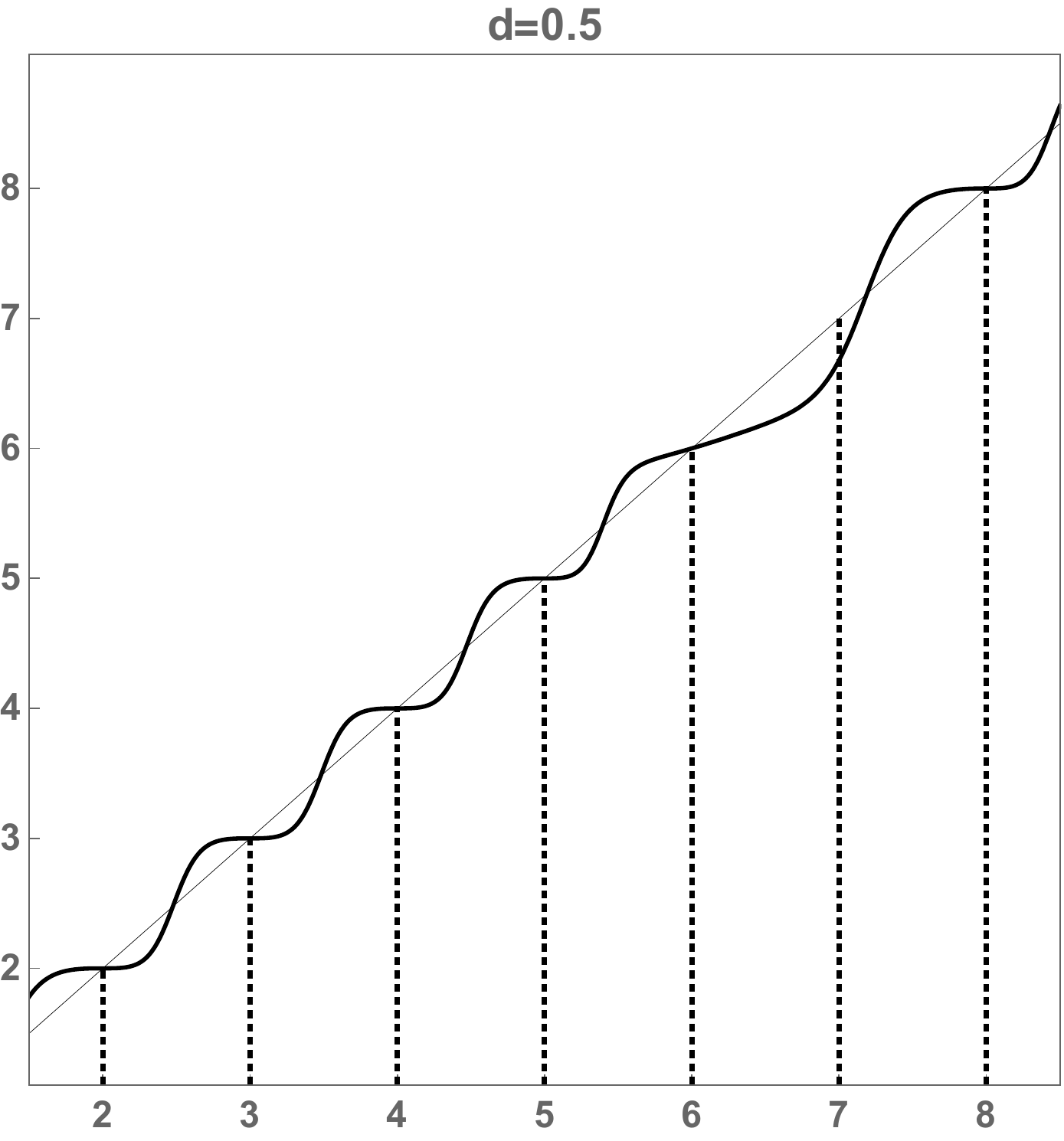} 
 \caption{ The    Halley educated  map $\widetilde{\Psi}_0$  for $d=0.5$ and $D=[1.5,8.5]$.\label{figAB}}
\end{center}
\end{figure}

\begin{figure}[hbt] 
\begin{center}
  \includegraphics[scale=0.30]{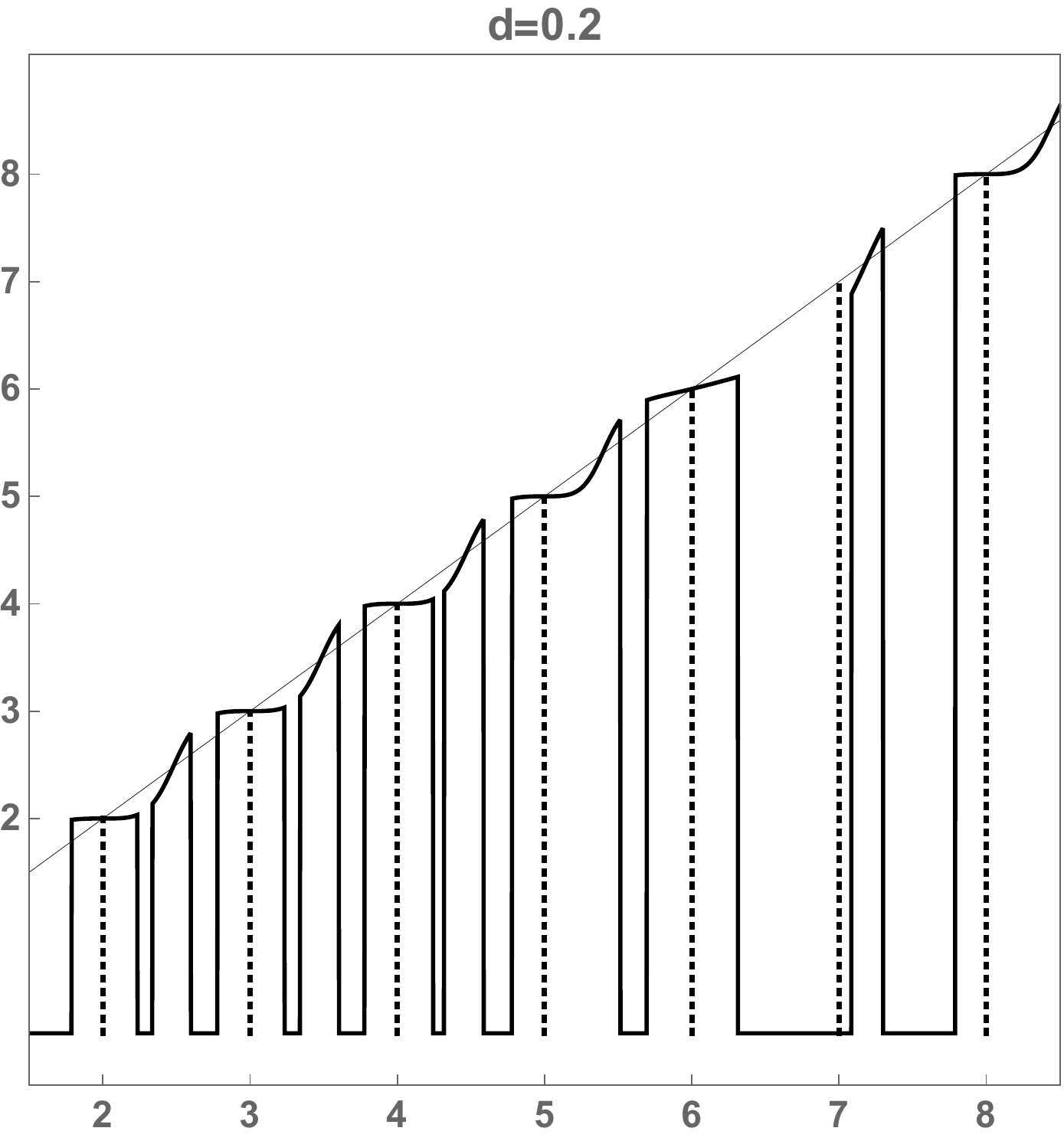}\,   \includegraphics[scale=0.30]{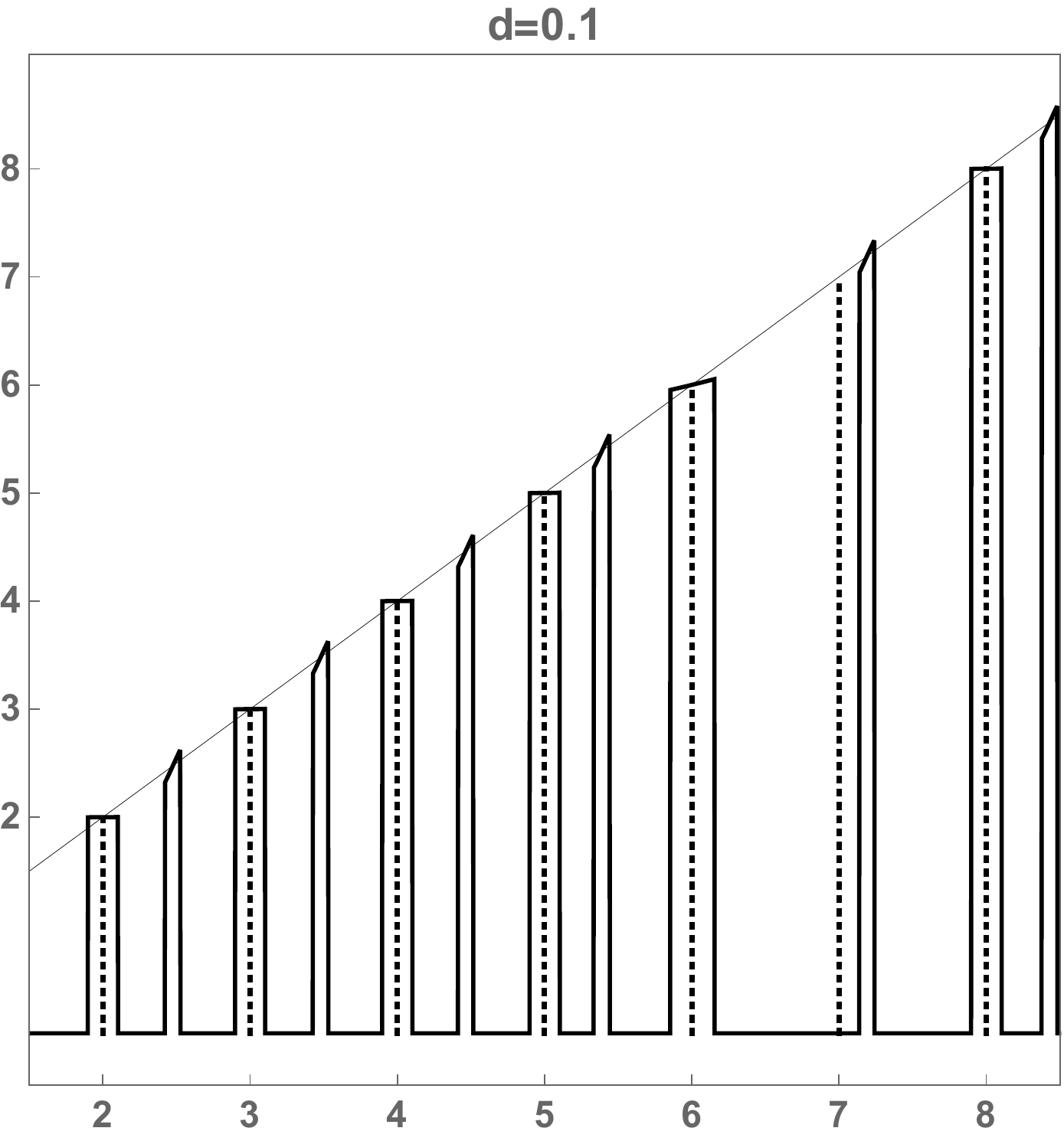}
 \caption{(A) Left,  the Halley educated map $\widetilde{\Psi}_0$ for $d=0.2$;(B) Right,   the Halley educated map $\widetilde{\Psi}_0$   for $d=0.1$. \label{figCD}}
\end{center}
\end{figure}

\noindent
Recall, from \eqref{es1} and \eqref{es2}, that in $D$ all the zeros of $f_3$ are simple except $z=6$ which is a double zero. So, $2,3,4,5,8$ are super attracting fixed points of  ${\cal H}_{f_3}$ while $z=6$ is only attracting.  
 Between two consecutive zeros of $f_3$ there is exactly   one repelling fixed point  of ${\cal H}_{f_3}$. This means that ${\cal H}_{f_3}$ has fixed points which are not roots of $f_3(x)=0$. Moreover, the fixed points of  ${\cal H}_{f_3}$ which are zeros of $f_3$ are precisely the attracting fixed points of  ${\cal H}_{f_3}$ as it can be observed in  Figure~\ref{figAB}. The resolution of the set of the attracting fixed points of ${\cal H}_{f_3}$  is $\eta=1$. By  Proposition~\ref{cor1},  if a vertical displacement $d$ is chosen such that $d< \eta/2=0.5$, then $\widetilde{\Psi}_0$ must separate the attracting fixed points of ${\cal H}_{f_3}$ (i.e. the zeros of $f_3$) in $D$. 
 This does not mean that $\widetilde{\Psi}_0$ separates all the fixed points of ${\cal H}_{f_3}$, as it is clear from Figure~\ref{figCD}-(A) where the graphic of $\widetilde{\Psi}_0$ for $d=0.2$ shows that  in the interval containing $5$ there are two fixed points of ${\cal H}_{f_3}$.

\noindent In  Figure~\ref{figAB} and Figure~\ref{figCD}-(B) we show the graphics of the quasi-step map $\widetilde{\Psi}_0$ obtained for the vertical displacement $d=0.5$ and $d=0.1$. We see that for   $d=0.5$, the educated map $\widetilde{\Psi}_0$  coincides with ${\cal H}_{f_3}$ while for $d=0.1$ all the fixed points of ${\cal H}_{f_3}$ (not just the attracting fixed ones)  have been separated in $D$.

 \begin{figure}[hbt] 
\begin{center}
 \includegraphics[scale=0.3]{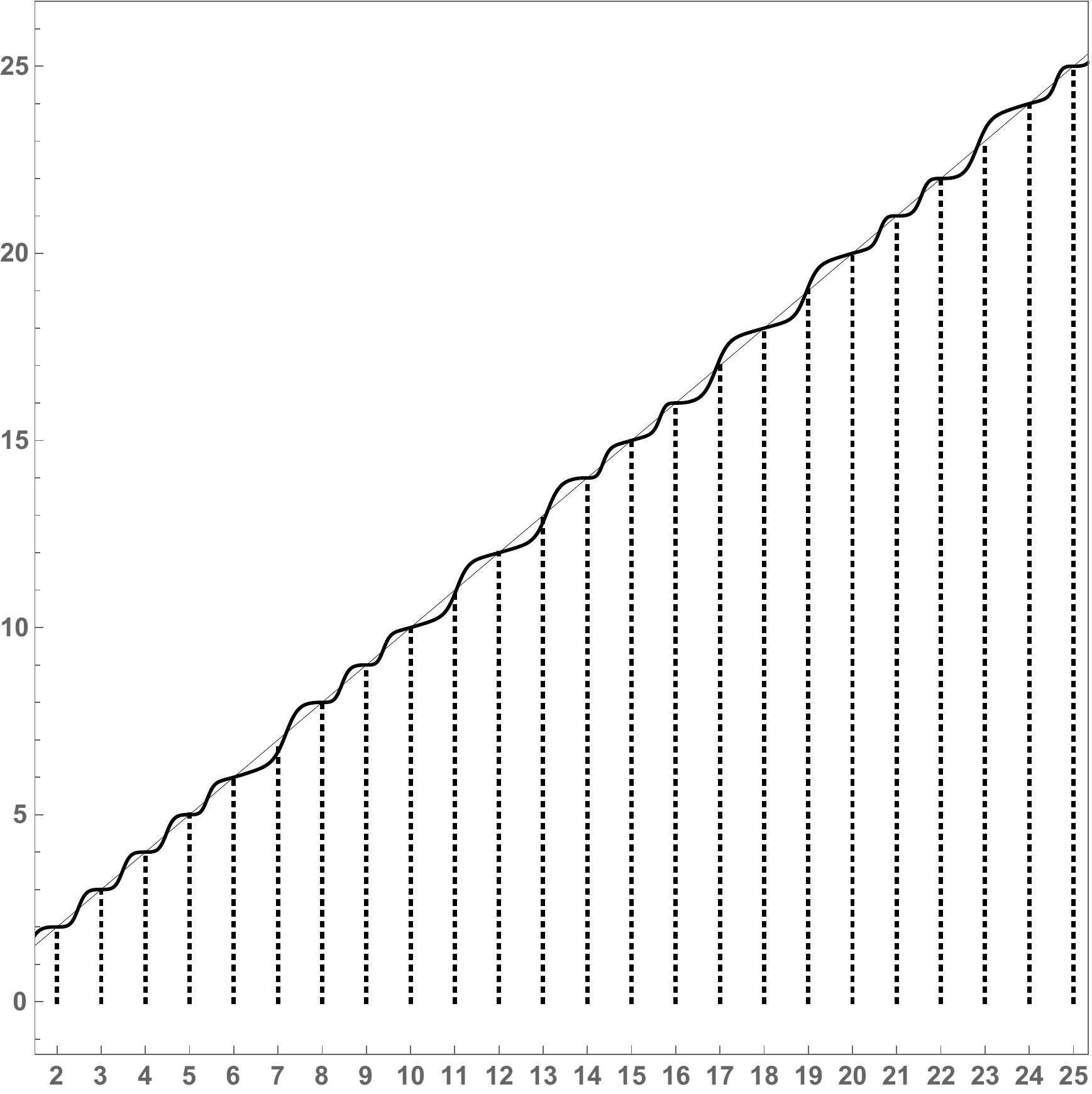}\, \includegraphics[scale=0.3]{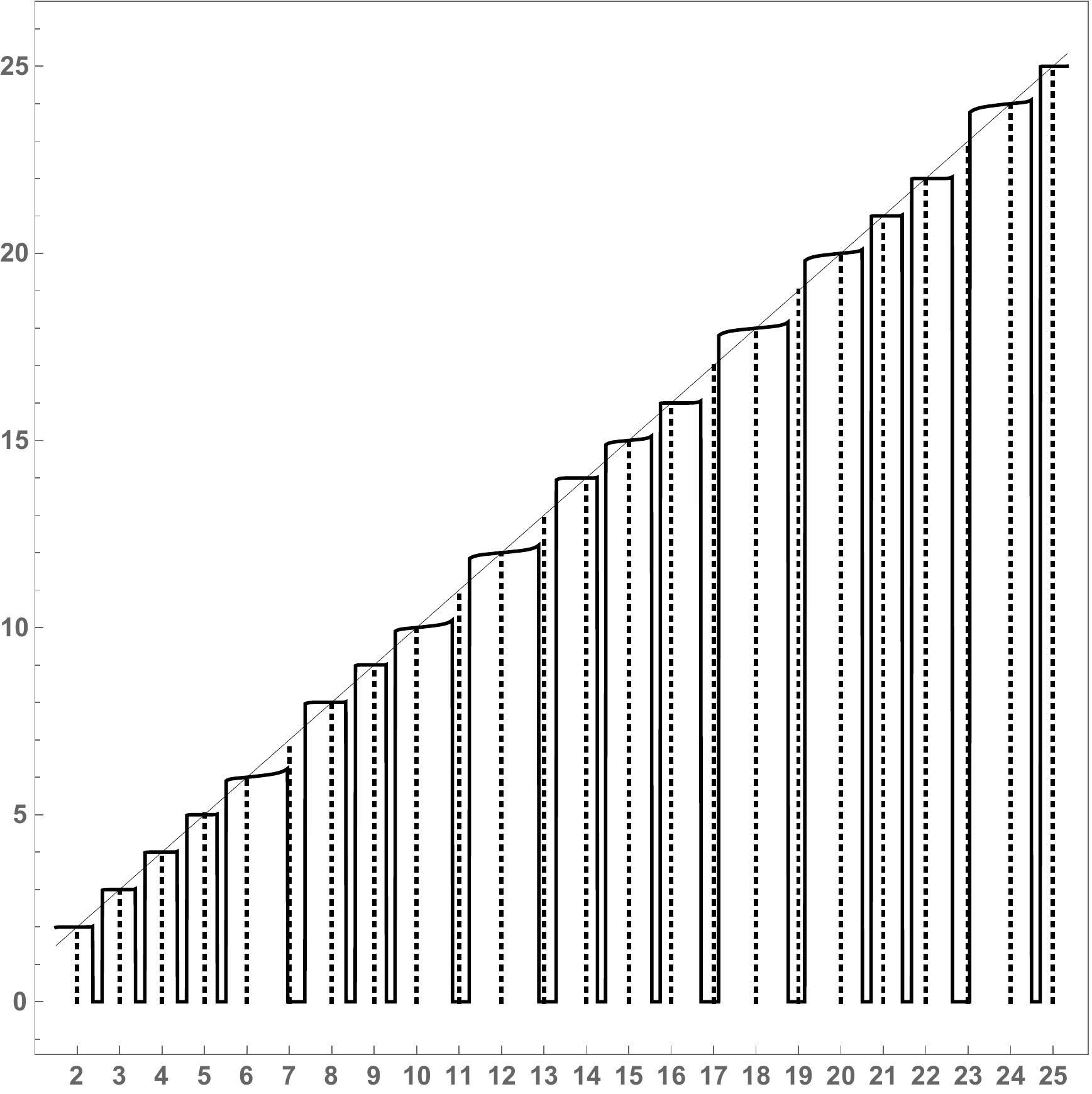}
 \caption{(A) Left the graphic of the Pruitt function $f_3$ in  $D=[1.5,25.5]$;(B) Right  the graphic of $\widetilde{\Psi}_1$ \label{HP1} separating all the zeros of $f_3$.}
\end{center}
\end{figure}

\medskip
\noindent We now consider $f_3$ defined in $D=[1.5, 25.5]$ and the corresponding Halley map ${\cal H}_{f_3}$. As before, the graphic of ${\cal H}_{f_3}$ in  Figure~\ref{HP1}-(A) clearly shows  the nature of the fixed points of this map. Namely, those fixed points of ${\cal H}_{f_3}$ that are roots of $f_3(x)=0$ are either acttracting or supper acttracting and the remaining fixed points are repelling  fixed points. The repelling fixed points  of ${\cal H}_{f_3}$ do not satisfy the predicate $\Pb_1$, since they are in contradiction with Lemma~\ref{propB}.  Therefore, a map $\widetilde{\Psi}_1$,  educated by $\Pb_1$,  separates all the roots of $f_3(x)=0$  (cf. Proposition~\ref{cor2}) as it is  clear from the graphics of $\widetilde{\Psi}_1$ displayed in Figure~\ref{HP1}-(B).

\medskip
\noindent
We remark that for certain classes of functions, such as the Pruitt  family $f_k$, the corresponding Halley map enjoys  the important property of continuity. This property is not verified in the case of the Newton map which will be treated in the next example. Moreover, the  analysis of the graphic of ${\cal H}_{f_3}$ in Figure~\ref{figAB} confirms what is expected concerning  the global convergence of Halley method, as discussed  in \cite{davies}, \cite{hernandez}.

\begin{figure}[hbt] 
\begin{center}
 \includegraphics[scale=0.40]{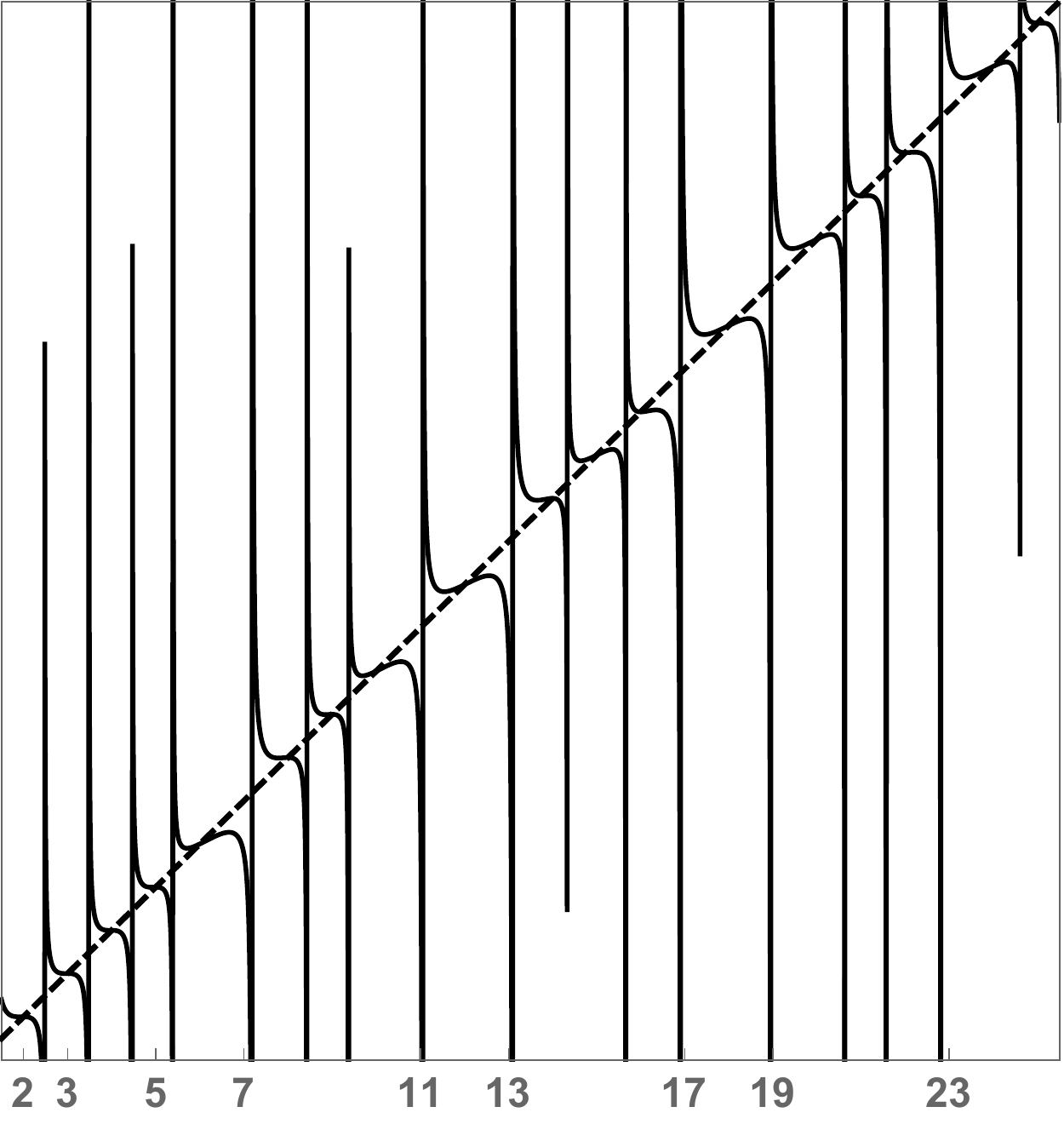}\,\,
  \includegraphics[scale=0.40]{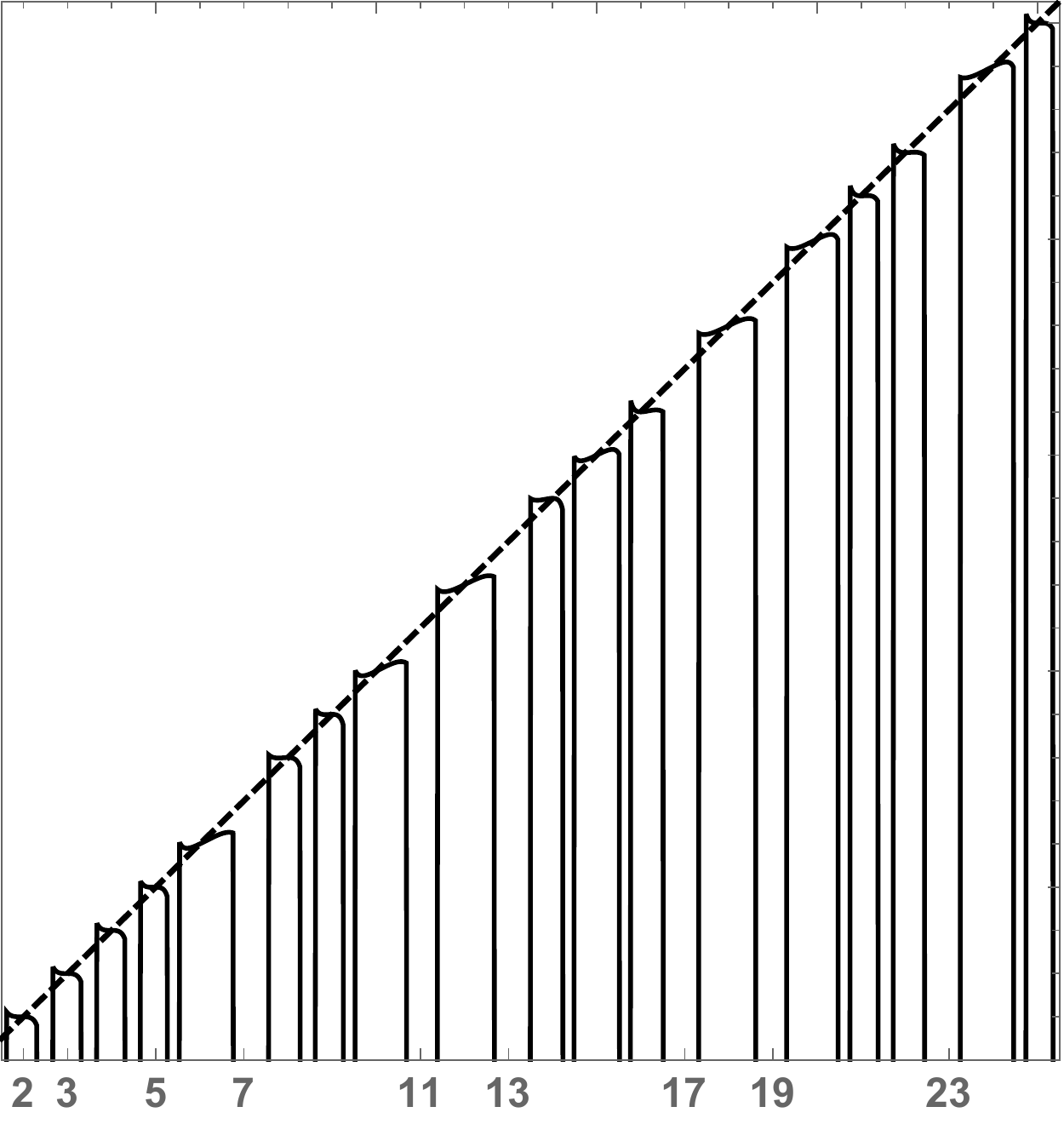}
 \caption{(A) Left, the Newton map ${\cal N}_{f_3}$ in  $D=[1.5,25.5]$;(B) Right, the  Newton educated  map $\widetilde{\Psi}_0$, for $d=0.5$. \label{fig2newt}}
\end{center}
\end{figure}

\begin{example}\label{exemplo3} (Newton quasi\--step maps for $f_3$)
\end{example}

\noindent
Let us consider the Newton map ${\cal N}_{f_3}$ corresponding to the  Pruitt function $f_3$ defined in $D=[1.5,25.5]$.  In contrast with the Halley function of the previous example,   ${\cal N}_{f_3}$ is not continuous in $D$ as it shows the graphic of  ${\cal N}_{f_3}$  in Figure~\ref{fig2newt}-(A). This graphic   presents several vertical asymptotes  which are due to the singularities of the function $f_3$ in the interval. However this time the fixed points of ${\cal N}_{f_3}$  coincide with the roots of $f_3(x)=0$,  that is ${\cal Z}={\cal P}_0\cup {\cal M}$ with ${\cal P}_0$ and ${\cal M}$ given by \eqref{rootsf3}. Obviously,  the resolution of ${\cal Z}$ is $\eta=1$ and all the fixed points of ${\cal N}_{f_3}$ satisfy the attractivity property referred in Proposition~\ref{cor1}. By this proposition,  if one chooses a vertical displacement $d<0.5$ all the roots of $f_3$ will be separated by the quasi-step map $\widetilde {\Psi}_0$ educated by the predicate $\Pb_0$. The  Figure~\ref{fig2newt}-(B) shows that the  ${\cal N}_{f_3}$ educated map $\widetilde{\Psi}_0$, for $d=0.5$, globally separates the zeros of $f_3$.

 \subsection{ Quasi\--step maps for a strongly oscillating function}\label{subexnum}

 \noindent
 The composition of quasi\--step maps can be particularly useful to compute highly accurate   roots of strongly oscillating functions as it is illustrated in the following numerical example.

  \begin{figure}[hbt] 
\begin{center}
  \includegraphics[scale=0.40]{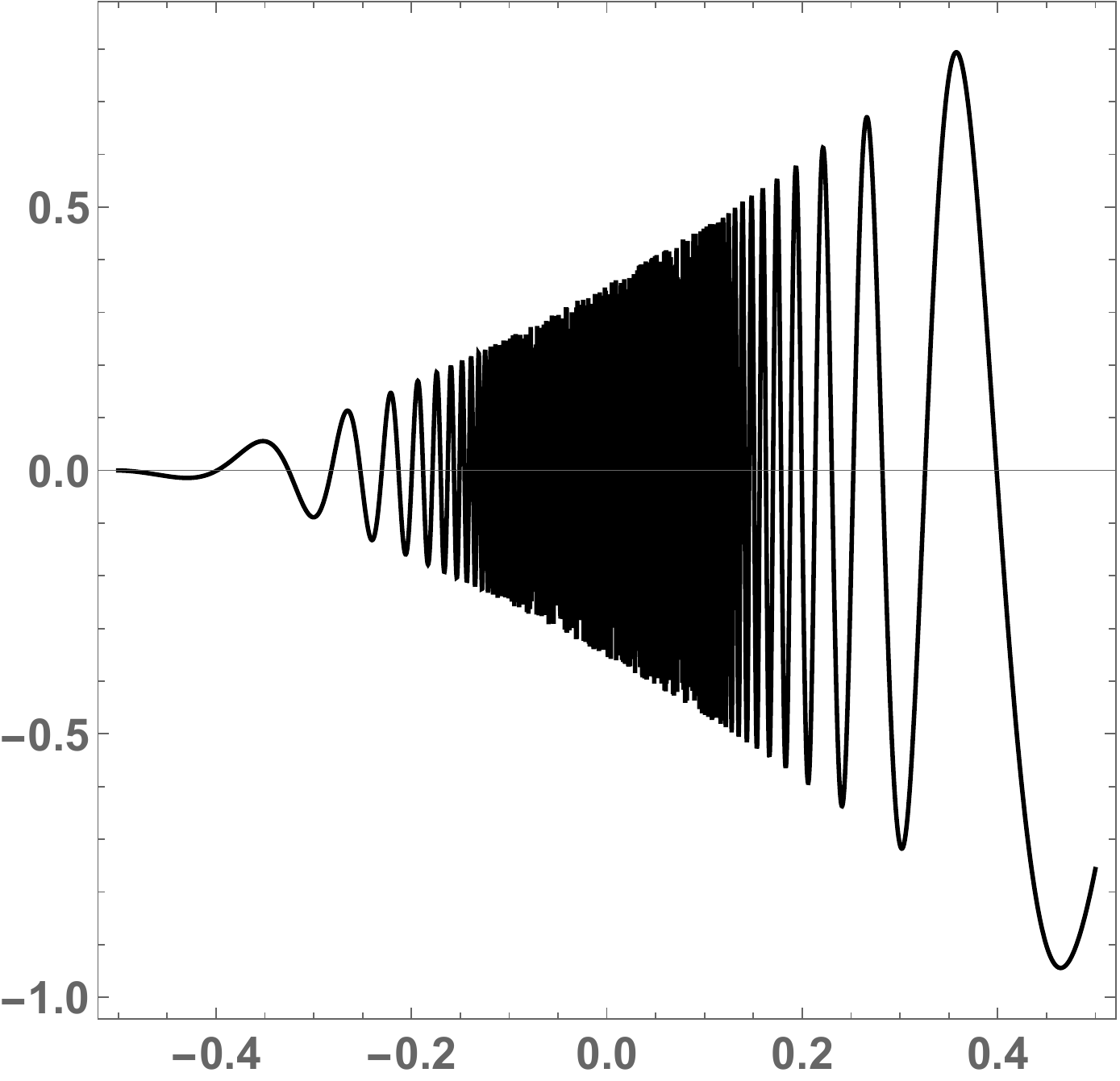}\;   \includegraphics[scale=0.40]{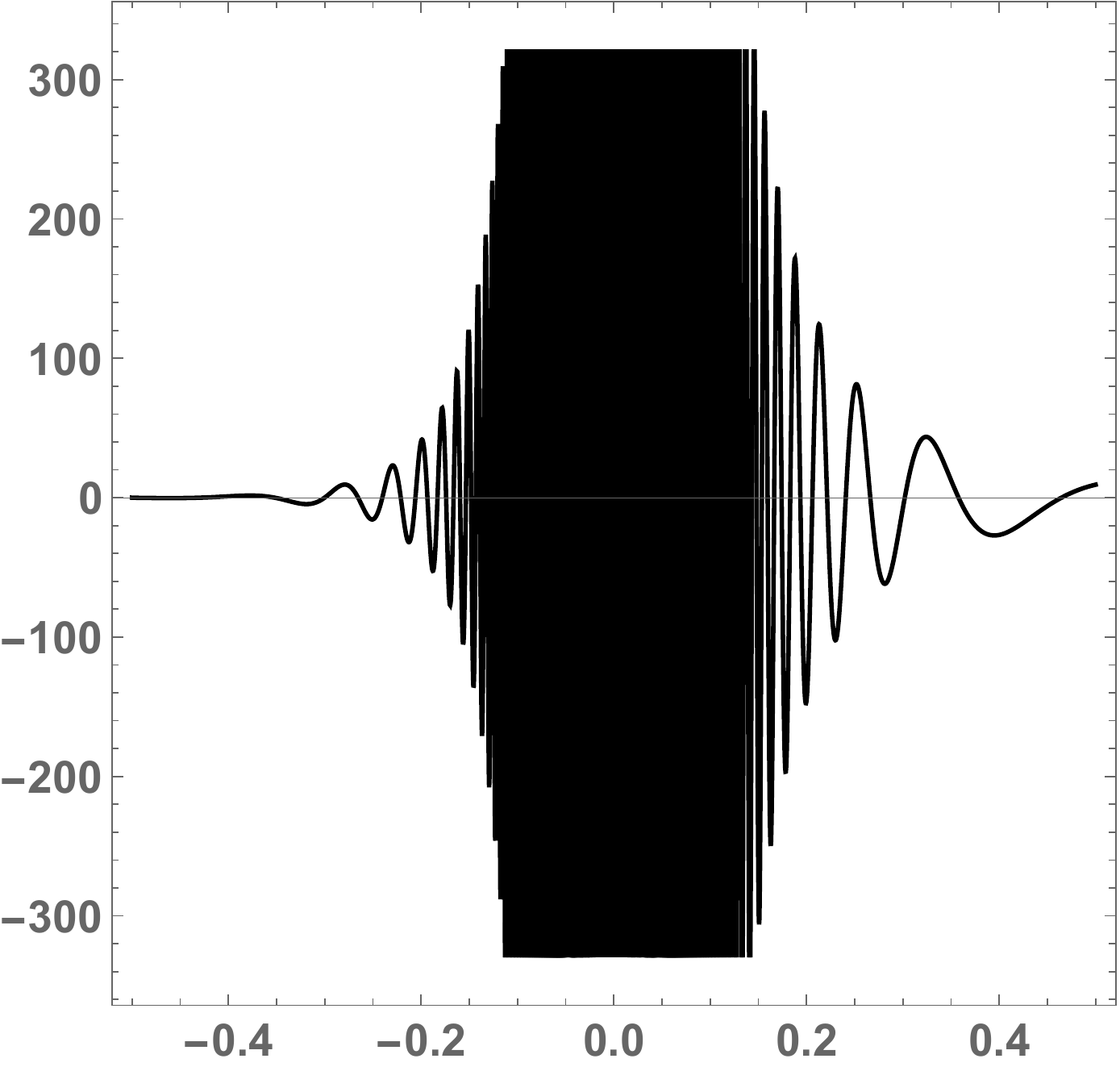}\quad
 \caption{Left, the strongly oscillating function $f$ and right its unbounded derivative. \label{figlips1}}
\end{center}
\end{figure}

\begin{example}\label{exemplo4} (Strongly oscillating function)
\end{example}

\noindent
We consider the function $f(0)=0$, $f(x)=(x+1/2)^{3/2}\sin(1/x)$,  for $x\neq 0$, defined in the interval $D=[-1/2,1/2]$. The function $f$ is differentiable with unbounded derivative and so it is not locally Lipschitzian. Close to the point $x=0$ the function is strongly oscillating  (see Figure \ref{figlips1}).

  \begin{figure}[hbt] 
\begin{center}
  \includegraphics[scale=0.4]{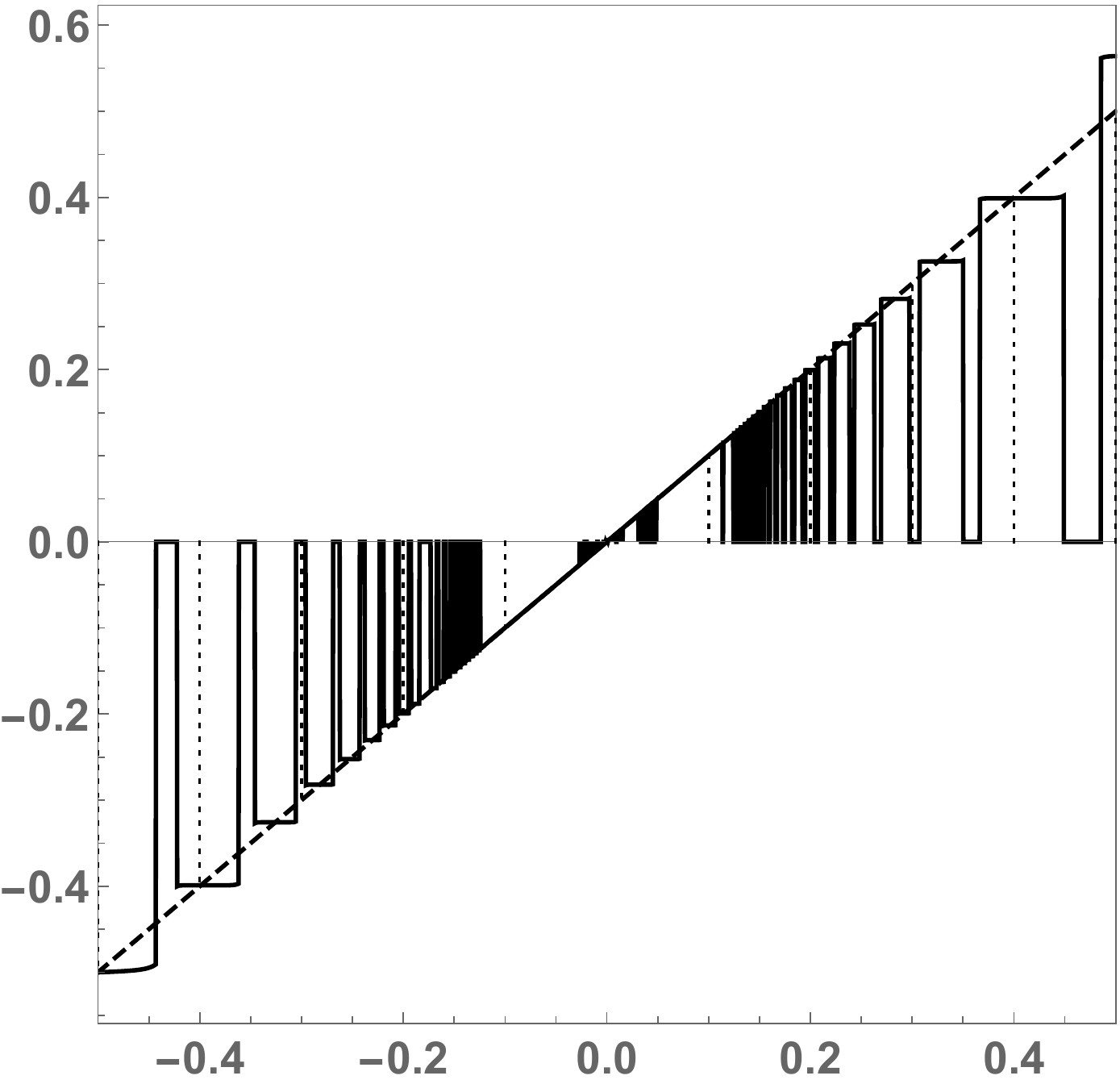}\;   \includegraphics[scale=0.4]{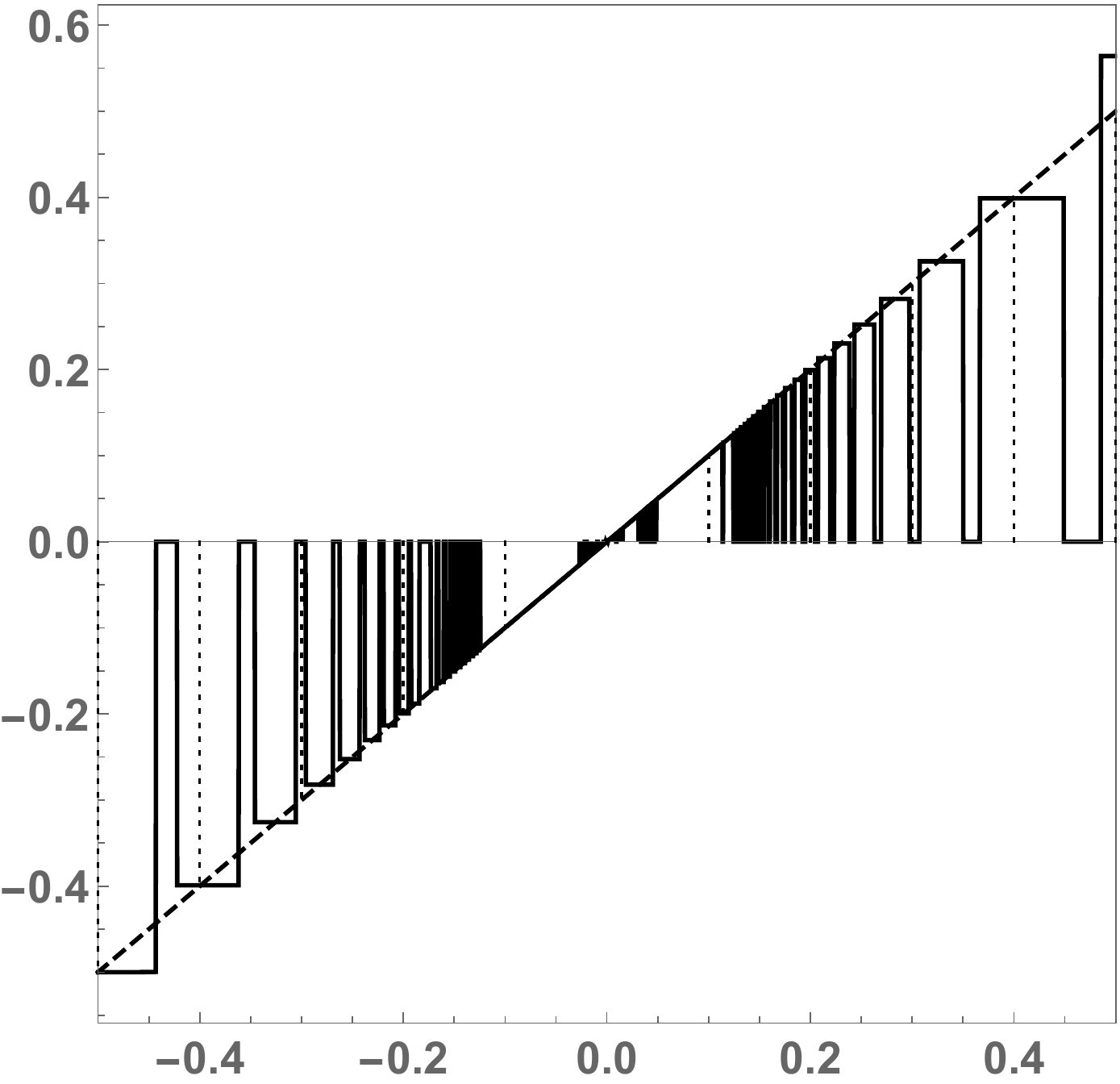}
 \caption{Halley educated maps $\widetilde {\cal H}_{f}$ (left) and $\widetilde {\cal H}^2_{f}$ (right). \label{figlips2}}
\end{center}
\end{figure}

\noindent
As before, let $ {\cal H}_{f}$ the Halley map associated to $f$ as defined in \eqref{eqhalley1}. The map $\widetilde {\cal H}_{f}$ denotes the  Halley map educated by the slope predicate $\Pb_1$ and  $\widetilde {\cal H}^r_{f}$ denotes its $r$-fold composition.
 
 \noindent
In Figure \ref{figlips2} it is displayed the graphics of the maps  $\widetilde {\cal H}_{f}$ and $\widetilde {\cal H}^2_{f}$ which show the quasi-step nature of these maps near the endpoints of the interval.  Moreover, the same  figure shows that in the subinterval say $I=[-0.2,0.2]$ there is  a great number of fixed points clustering around the origin.


\medskip
\noindent
We used the system {\sl Mathematica} to compute approximations of the fixed points of $\widetilde {\cal H}_{f}$  in a certain interval centered at the origin. All the computations  were carried out  on  a $2GHz$ Intel Core $i7$ personal computer, using standard double precision floating point arithmetic.

   \begin{table}[htbp]
$$\begin{array} {|c|c|} \hline
s&\mbox{CPU time} \\ \hline
2&0.38\\  \hline
100&16.3\\  \hline
250&40.5\\  \hline
500&80.5 \\ \hline
\end{array}$$
\caption{CPU time (in seconds) to run $FindInstance$ to simultaneously compute a sample of $s$ roots  in the interval $J=[-0.0097, 0.0097]$. \label{tabela1}}
\end{table}

\noindent
Our  computations  were compared with those obtained by using  some built-in routines of the system {\sl Mathematica}. In particular, the command  $Find\-Instance[expr,vars,\-Reals,s]$, with $vars$ and $expr$ assumed to be real, was used to compute an (unsorted) sample of $s$ real roots of $f(x)=0$. Notably,  in the interval   $J=[x_{min},x_{max}]=[-0.0097, 0.0097]$ the code line $FindInstan\-ce[\{f[u] == 0.,\,  xmin \leq u \leq  xmax\}, u, s]$ was ran for several values of $s$. The respective CPU time (in seconds) is given  in Table~\ref{tabela1}, which shows that for large $s$ the CPU time approximately doubles with $s$. Unless one uses the system option  $WorkingPrecision$, Table~\ref{tabela1} shows that for $s=200$ zeros  of $f$ in the interval $J$, the referred time of $80.5$ seconds is quite unacceptable. In contrast,  using the same default double precision computer arithmetic,  the Halley educated  map $\widetilde {\cal H}_{f}$ can separate, in less than one second,   $s=200$  fixed points in the interval  $J$.

  \noindent
We considered discretized versions of the maps $\widetilde {\cal H}_{f}$, $\widetilde {\cal H}^2_{f}$ and $\widetilde {\cal H}^3_{f}$ obtained by dividing the interval $J=[x_{min}, x_{max}]=[-0.0097, 0.0097]$ into $N=1500$ parts of length $h=(x_{max}-x_{min})/N$. The respective $N+1$ values of the maps at the points $x_i=x_{min}+i\, h$, with $i=0,\ldots,N$ were tabulated. Due to the large number of points considered,  Table~\ref{tabela2} only shows some of these values, namely those which are near the endpoints of $J$.  For convenience, we will call $\cal T$  the  full set of data obtained.

  \begin{table}[htbp]
 \centering
 \small
   \tabcolsep=0.10cm
 \begin{tabular}{|p{2.0cm}| p{2.0cm}|p{2.0cm}|p{2.0cm}|}\hline
 $x_i$& $\widetilde { \vspace{2mm}\cal H}_{f}(x_i)$& $\widetilde {\cal H}^2_{f}(x_i)$&$\widetilde {\cal H}^3_{f}(x_i)$\\ \hline
 -0.0097&-0.0097000521& -0.0097000521& -0.0097000521\\ \hline
-0.0096870667& -0.0096871749& -0.0096871749& -0.0096871749\\  \hline
-0.0096741333& -0.0096743489& -0.0096743489&-0.0096743489\\  \hline
-0.0096612& -0.0096615735&-0.0096615738& -0.0096615738\\  \hline
-0.0096482667& 0& 0&0\\  \hline
-0.0096353333& 0& 0& 0\\  \hline
-0.0096224& -0.0096221501& -0.0096221501& -0.0096221501\\\hline
-0.0096094667& -0.0096095803& -0.0096095803&-0.0096095803\\ \hline
$\qquad \cdots$&$\qquad \cdots$&$\qquad \cdots$&$\qquad \cdots$\\  \hline
$\qquad \cdots$&$ \qquad \cdots$&$ \qquad \cdots$&$ \qquad \cdots$\\ \hline
0.0096094667& 0.0096095803& 0.0096095803& 0.0096095803\\ \hline
0.0096224& 0.0096221501& 0.0096221501& 0.0096221501\\ \hline
0.0096353333& 0& 0& 0\\ \hline
0.0096482667& 0& 0& 0\\ \hline
0.0096612& 0.0096615735& 0.0096615738& 0.0096615738\\ \hline
0.0096741333& 0.0096743489& 0.0096743489& 0.0096743489\\ \hline
0.0096870667& 0.0096871749& 0.0096871749& 0.0096871749\\\hline
0.0097& 0.0097000521& 0.0097000521& 0.0097000521\\ \hline
 \end{tabular}
 \caption{Educated maps $\widetilde {\cal H}^r_{f}(x)$, for $r=1,2, 3$.\label{tabela2} }
 \end{table}
 
\medskip
\noindent
 It is clear from Table~\ref{tabela2} that the points $x_i$ where  the value zero occurs for all the maps $\widetilde {\cal H}_{f}$, $\widetilde {\cal H}^2_{f}$ and $\widetilde {\cal H}^3_{f}$ provide a collection of the subintervals where the slope predicate holds true. In the full set ${\cal T}$ we found  $273$ of such subintervals. In each of these subintervals there exists at least a fixed point of $\widetilde {\cal H}_{f}$. 
 The analysis of ${\cal T}$ also shows that the maps $\widetilde {\cal H}^3_{f}$ and $\widetilde {\cal H}^2_{f}$  are  numerically invariant   (see also  Table~\ref{tabela2}) and so all the computed nonzero values $\widetilde {\cal H}^2_{f} (x_i)$ are approximations of fixed points of the Halley map  ${\cal H}_{f}$, with eight significant decimal digits.
 
 \noindent
 The previous  procedures may be implemented in order to obtain high precision approximations of the fixed points of ${\cal H}_f$. This can be achieved by considering  in the computations not only a convenient machine precision but also an appropriate $r$-fold composition of $\widetilde {\cal H}_{f}$. For instance, taking the same sample of  1501 points $x_i$  in $J$, an extended precision of 1000 decimal digits,  and  computing $\widetilde {\cal H}^r_{f}(x_i)$, for $r=2$ to $r=5$, we obtained a new table  of data in about 10 seconds of CPU time. In this case one can verifies that $\widetilde {\cal H}^4_{f}$ and $\widetilde {\cal H}^5_{f}$ are numerically invariant.
In particular, for the point $x_{N-1}=0.0097-h\simeq 0.0096870667$ we have  $z_{N-1}=\widetilde {\cal H}^5_{f}(x_{N-1})$ with

\begin{equation}\label{valor}
\begin{array}{ll}
z_{N-1}&=0.00968717492126197578990697768061822598\\
& \hspace{4cm}\vdots\\
& 3785734861841858811484051025080640392894,
\end{array}
\end{equation}
where only a certain number of the initial and the final  1000 decimal machine digits are displayed. Computing the residual $f(z_{N-1})$ we obtain 
$$f(z_{N-1}) \simeq 0.*10^{-997}.$$

\noindent
In order to check that $z_{N-1}$ is in fact an accurate root of $f(x)=0$, we use the  {\sl Mathematica} function $FindRoot$ as follows: 
 $$
\begin{array}{l}
x_{N-1} = 0.0096870667;\\
 z = x\, /.\, FindRoot[f[x] == 0, \{x, x_{N-1}\},\, WorkingPrecision \,->\, 1000]; 
 \end{array}
$$
The respective value of $z$ is such that  $z-z_{N-1}\simeq 0.*10^{-1002}$ which shows that all the 1000 decimal digits of $z_{N-1}$, in \eqref{valor}, are correct. 

 \medskip
\noindent
As a final remark let us refer that we have applied  with success several Newton and Halley educated maps to a battery of test functions suggested  in \cite{nerinckx,chun,wu,galantai}. Suitable discretized versions of the respective quasi\--step maps enable the computation of high accurate roots regardless these roots are simple or multiple, and so the approach seems to be particularly useful, in particular for the global separation of zeros of strongly oscillating functions.

\small{

}

\end{document}